\documentclass[12pt]{amsart}
\usepackage{fullpage,url,amssymb}
\usepackage[OT2,T1]{fontenc}

\usepackage{color}
\usepackage[
        colorlinks, citecolor=darkgreen,
        backref,
        pdfauthor={Alain Kraus}, 
]{hyperref}
\usepackage{comment}

\newcommand{\sym}{\mathbb{S}}

\newcommand{\Q}{\mathbb{Q}}

\newcommand{\Z}{\mathbb{Z}}

\newcommand{\N}{\mathbb{N}}

\newcommand{\calL}{\mathcal{L}}

\numberwithin{equation}{section}

\newtheorem{theorem}{Theorem}
\newtheorem{lemma}{Lemma}
\newtheorem{corollary}{Corollary}
\newtheorem{proposition}{Proposition}

\theoremstyle{definition}

\theoremstyle{remark}

\definecolor{darkgreen}{rgb}{0,0.5,0}

\begin{document}

\title{Quartic points on the Fermat quintic}

\author{Alain Kraus}

\date{\today}

\begin{abstract}
In  this paper, we study  the algebraic points  of degree $4$ over $\Q$ on  the Fermat curve  $F_5/\Q$ of equation $x^5+y^5+z^5=0$. A geometrical description of these points has been given in 1997 by Klassen and Tzermias. Using their result,  as well as Bruin's work about   diophantine equations of signature $(5,5,2)$, we 
give here  an algebraic description of these points. In particular, we prove there is only one Galois extension of $\Q$ of degree $4$ that arises as  the field of definition of a non-trivial point of $F_5$.
\end{abstract}


\maketitle
\section{Introduction}
Let us denote by $F_5$ the quintic Fermat curve over $\Q$ given by the equation 
$$x^5+y^5+z^5=0.$$
Let $P$ be a point in $F_5(\overline \Q)$. 
The degree of $P$  is the degree of its field of definition over $\Q$. Write $P=(x,y,z)$ for the projective coordinates of $P$. It is said to be 
 non-trivial if $xyz\neq 0$. 
Let $\zeta$ be a primitive cubic root of unity and
$$a=(0,-1,1),\quad b=(-1,0,1),\quad c=(-1,1,0),\quad w=(\zeta,\zeta^2,1), \quad \overline{w}=(\zeta^2,\zeta,1).$$
It is well known that $F_5(\Q)=\left\lbrace a,b,c\right\rbrace$. In  1978,  Gross and Rohrlich have proved that the only quadratic points of $F_5$  
are $w$ and $\overline w$  \cite[th. 5.1]{GR}. In 1997, by proving that the group of  $\Q$-rational points of the Jacobian of $F_5$ is  isomorphic to $(\Z/5\Z)^2$, 
and by expliciting  generators, 
Klassen and Tzermias have described geometrically all the points of $F_5$ whose degrees are less than  $6$  \cite[th. 1]{KlassenTzermias}. I mention that Top and Sall have  pushed  further this description for points of $F_5$ of degrees less than $12$  \cite{TopSall}. In particular, Klassen and Tzermias have established the following statement  :

\begin{theorem}  
The points of degree $4$  of $F_5$ arise as the intersection of $F_5$ with a rational line passing through $a,b$ or $c$.
 \end {theorem}

Using this result,  and Bruin's work  about the diophantine equations $16x^5+y^5=z^2$ and $4x^5+y^5=z^2$  \cite{Bruin, Ivorra}, we propose in this paper   to give an algebraic description of the non-trivial quartic points of $F_5$.

\section{Statement of the results}

 Let $K$ be a number field of degree $4$ over $\Q$.

\begin{theorem}  
Suppose that $F_5(K)$ has a non-trivial point of degree $4$. One of the following conditions is satisfied :

1) the Galois closure of $K$ is a dihedral  extension of $\Q$ of degree $8$.

2) One has  
\begin{equation} 
K=\Q(\alpha)\quad \hbox{with}\quad 31\alpha^4-36\alpha^3+26\alpha^2-36\alpha+31=0.
\end{equation}
The extension $K/\Q$ is   cyclic. Up to Galois conjugation and permutation, $(2,2\alpha,-\alpha-1)$
is the only  non-trivial point  in $F_5(K)$.
 \end {theorem}

As a direct consequence of \cite[th. 5.1]{GR} and the previous theorem, we obtain :

\begin{corollary} Suppose that $K$ does not satisfy one of the two conditions above. The set of non-trivial points of 
$F_5(K)$ is contained in $\left\lbrace w,\overline w\right\rbrace$.
\end{corollary} 

All that follows is devoted to the proof of theorem 2.

\section{Preliminary results}
Let $P=(x,y,z)\in F_5(K)$ be a non-trivial point of degree $4$. By permuting $x,y,z$ if necessary, we can  suppose that $P$ belongs to a $\Q$-rational line $\calL$ passing through 
$a=(0,-1,1)$ (th. 1).  Moreover, $P$ being non-trivial we shall assume 
\begin{equation} 
z=1.
\end{equation}

\begin{lemma} 
One has $K=\Q(y)$. 
There exists $t\in \Q$, $t\neq -1$, such that
\begin{equation} 
y^4+uy^3+(u+2)y^2+uy+1=0\quad \hbox{with}\quad u=\frac{4t^5-1}{t^5+1},    
\end{equation} 
\begin{equation} 
 x=t(y+1).
\end{equation} 

\end{lemma}

\begin{proof} 
The equation of the tangent line    to $F_5$ at the point $a$ is $Y+Z=0$. Since   $x\neq 0$,   it is distinct from $\calL$.
According to (3.1), it follows  there exists $t\in \Q$ such that  
$$x=t(y+1).$$
In particular, one has $K=\Q(y)$. 
Furthermore, one has 
\begin{equation} 
t\neq -1.
\end{equation} 
 Indeed, if $t=-1$, the equalities
$x+y+1=0$ and $x^5+y^5+1=0$
imply
$$x(x+1)(x^2+x+1)=0.$$
Since $P$ is non-trivial, one has $x(x+1)\neq 0$, so $x^2+x+1=0$. This leads to $P=w$ or $P=\overline w$, which contradicts the fact that  $P$ is not a quadratic point, and proves (3.4).  

From  the equalities (3.3) and $x^5+y^5+1=0$, as well as  the condition $y\neq -1$,  we then deduce the lemma.
\end{proof} 

Let  $G$  be the Galois group of the Galois closure of $K$ over $\Q$. Let us denote by $|G|$ the order of $G$. 

\begin{lemma} 
1) One has $|G|\in \left\lbrace 4,8\right\rbrace$.

2) Suppose that $|G|=4$.  One of the  two following conditions is satisfied  :
\begin{itemize}
 \item[(1)]   $5(16t^5+1)$ is a square in $\Q$.
  \item[(2)]  $(1-4t^5)(16t^5+1)$ is a square in $\Q$.
 \end{itemize}

\end{lemma}

\begin{proof} Let us denote in $\Q[X]$  
$$f=X^4+uX^3+(u+2)X^2+uX+1.$$
One has $f(y)=0$ (lemma 1). 
Let $\varepsilon\in \overline \Q$  such that
$$\varepsilon^2=u^2-4u.$$
The element  $y+\frac1y$ is a root of the polynomial $X^2+uX+u$. So we have the inclusion
\begin{equation}
\Q(\varepsilon) \subseteq K.
\end{equation}
Moreover, we  have 
 the equality
\begin{equation}
f=\left(X^2+\frac{u-\varepsilon}{2}X+1\right)\left(X^2+\frac{u+\varepsilon}{2}X+1\right).
\end{equation}
Since $K=\Q(y)$ and $[K:\Q]=4$,  
 we  have
\begin{equation}
[\Q(\varepsilon):\Q]=2.
\end{equation}
From (3.6), we deduce that 
the  roots of 
$f$ belong to at most two  quadratic extensions of $\Q(\varepsilon)$.  The equality (3.7) then implies  $|G|\leq 8$.  Since $4$ divides $|G|$, this proves the first assertion.

Henceforth let us suppose  $|G|=4$ i.e. the extension $K/\Q$ is Galois.
Let $\Delta$ be the  discriminant  of $f$. One has the equalities
\begin{equation}
\Delta=-u^2(u-4)^3(3u+4)=5^3\frac{(4t^5-1)^2(16t^5+1)}{(t^5+1)^6}.
\end{equation}
Let us prove that 
\begin{equation}
\Delta\ \hbox{is a square in }\ \Q(\varepsilon).
\end{equation}

From (3.6) and our assumption,  the roots of the polynomials 
$$X^2+\frac{u-\varepsilon}{2}X+1\quad \hbox{and}\quad  X^2+\frac{u+\varepsilon}{2}X+1$$ belong to $K$, which is a quadratic extension of $\Q(\varepsilon)$ ((3.5) and (3.7)). Therefore, the  product of their discriminants 
$$\left(\left(\frac{u-\varepsilon}{2}\right)^2-4\right)\left( \left(\frac{u+\varepsilon}{2}\right)^2-4\right)\quad \hbox{i.e.}\quad -(u-4)(3u+4)$$
must be a square in $\Q(\varepsilon)$. The first equality of (3.8) then implies  (3.9).

Suppose that the condition (1) is not satisfied. From second equality of (3.8),   we deduce that $\Delta$ in not a square in $\Q$.  
It follows from (3.9) that we have 
$$\Q\left(\sqrt{\Delta}\right)=\Q(\varepsilon).$$
Therefore, $\Delta(u^2-4u)$ is a square in $\Q$, in other words,  such is the case for $-u(3u+4)$. One has the equality
$$-u(3u+4)=\frac{(1-4t^5)(16t^5+1)}{(t^5+1)^2}.$$
This implies the condition (2) and proves the lemma.
 \end{proof} 

\section{The curve $C_1/\Q$}
Let us denote by $C_1/\Q$ the curve, of genus $2$, given by the  equation
$$Y^2=5(16X^5+1).$$

\begin{proposition} 
The set  $C_1(\Q)$ is empty.
\end{proposition} 

\begin{proof}
Suppose there exists a point $(X,Y)\in C_1(\Q)$. Let $Z=\frac Y5$. We obtain
\begin{equation}
5Z^2=16X^5+1.
 \end{equation}
 
Let $a$ and $b$ be coprime integers, with $b\in \N$, such that
$$X=\frac{a}{b}.$$
Let us prove there exists $c\in \N$ such that
\begin{equation}
b=5c^2.
\end{equation}
For every  prime number $p$, let $v_p$ be the $p$-adic valuation over $\Q$. 

If $p$ is a prime number dividing $b$, distinct from $2,5$, one has
$2v_p(Z)=-5v_p(b)$, consequently
\begin{equation}
v_p(b)\equiv 0 \mod 2.
 \end{equation}
Moreover, one has $v_2(X)<0$ ($5$ is not a square  modulo $8$), so $4-5v_2(b)=2v_2(Z)$. In particular,  one has
\begin{equation}
v_2(b)\equiv 0 \mod 2.
 \end{equation}
Let us verify the congruence
\begin{equation}
v_5(b)\equiv 1 \mod 2.
 \end{equation}
One has $v_5(X)\leq 0$. Suppose $v_5(X)=0$. In this case, one has $X^5\equiv \pm 1, \pm 7 \mod 25$. The equality (4.1)  implies $X^5\equiv -1 \mod 25$ and 
$Z^2\equiv 2 \mod 5$, which leads to a contradiction. Therefore, we have $1+2v_5(Z)=-5v_5(b)$, which proves   (4.5).

The conditions (4.3), (4.4) and (4.5) then  imply (4.2).

We deduce from (4.1) and (4.2) the equality
$$16a^5+b^5=d^2\quad \hbox{with}\quad d=5^3c^5Z.$$
One has $ab\neq 0$. From the  informations given in  the Appendix of \cite{Ivorra}, this implies 
$$(a,b,d)=(-1,2,\pm 4).$$
We obtain $X=-1/2$, which is not the  abscissa of a point of $C_1(\Q)$, hence the result.
\end{proof}

\section{The curve $C_2/\Q$}
Let us denote by $C_2/\Q$ the curve, of genus $4$, given by the  equation
$$Y^2=(1-4X^5)(16X^5+1).$$

\begin{proposition} 
One has
$$C_2(\Q)=\left\lbrace (0,\pm 1), (-1/2,\pm 3/4)  \right\rbrace.$$
\end{proposition} 

\begin{proof}
Let $(X,Y)$ be a point of $C_2(\Q)$. Let $a$ and $b$ be coprime integers  such that
$$X=\frac{a}{b}.$$
We obtain the equality
\begin{equation}
(Yb^5)^2=(b^5-4a^5)(16a^5+b^5).
 \end{equation}
Therefore, $(b^5-4a^5)(16a^5+b^5)$ is the square of an integer. Moreover, $b^5-4a^5$ and $16a^5+b^5$ are coprime apart from  $2$ and $5$. 
So, changing $(a,b)$ by $(-a,-b)$ if necessary, there exists $d\in \N$ such that
$$b^5-4a^5\in \left\lbrace d^2, 2d^2, 5d^2, 10d^2\right\rbrace.$$

Suppose $b^5-4a^5\in \left\lbrace 2d^2, 10d^2\right\rbrace.$ In this case, $b$ must be even, therefore $v_2(2d^2)=2$, which  is not.

Suppose   $b^5-4a^5=d^2$. One has $b\neq 0$.  It then comes from \cite{Ivorra} that
$$a=0\quad \hbox{or}\quad (a,b,d)=(-1,2,\pm 6).$$
We obtain $X=0$ or $X=-1/2$, which leads to the announced points in the statement.

Suppose  $b^5-4a^5=5d^2$. It follows from (5.1) that there exists $c\in \N$ such that $16a^5+b^5=5c^2$. Since $a$ and $b$ are coprime, $5$ does not divide $ab$.  We then directly verify that the two equalities $b^5-4a^5=5d^2$ and $16a^5+b^5=5c^2$ do not have simultaneously any solutions modulo $25$, hence the result.
\end{proof}

\section{End of the proof of Theorem 2}
The group $G$ is isomorphic to a subgroup of   the symmetric group $\sym_4$ and one has $|G|=4$ or $|G|=8$ (lemma 2).
In case $|G|=8$,  $G$ is isomorphic to a $2$-Sylow subgroup of $\sym_4$, that is dihedral. 

Suppose $|G|=4$ and let us prove the assertion 2 of the theorem. 

First, we directly verify that the extension $K/\Q$ defined by the condition (2.1) is cyclic of degree $4$, and that the point $(2,2\alpha,-\alpha-1)$ belongs to $F_5(K)$. 

Conversely,  from the proposition 1, the condition (1) of the lemma 2 is not satisfied. The condition (2) and the proposition 2 imply that 
$t=0$ or $t=-1/2$. The case $t=0$ is excluded because $P$ is non-trivial. With  the condition (3.2),  we obtain 
$$u=-\frac{36}{31}.$$
Thus, necessarily $y$ is a root of the polynomial
$31X^4-36X^3+26X^2-36X+31$, in other words $y$ is a conjugate over $\Q$ of $\alpha$. The equality  (3.3),
$$x=-\frac{y+1}{2}$$
then implies  the result.

\bigskip
{\bf{Acknowledgments.}} I thank D. Bernardi for his remarks during the writing of this paper.

\bigskip\bigskip
Alain Kraus : alain.kraus@imj-prg.fr

Universit\'e de Paris VI, Institut de Math\'ematiques de Jussieu, 4 place Jussieu, 75005 Paris, France.

\end{document}